\documentclass[12pt]{amsart}
\usepackage[utf8]{inputenc}

\usepackage{mathrsfs} 
\usepackage{amssymb}
\usepackage{bm}
\usepackage{graphicx}
\usepackage[centertags]{amsmath}
\usepackage{amsfonts}
\usepackage{amsthm}
\usepackage{enumerate}
\usepackage{tocvsec2}
\usepackage{xcolor}
\usepackage[margin=1in]{geometry}

\newtheorem{theorem}{Theorem}[section]
\newtheorem*{theorem*}{Theorem}

\newtheorem{corollary}[theorem]{Corollary}
\newtheorem{lemma}[theorem]{Lemma}

\newtheorem{proposition}[theorem]{Proposition}

\newtheorem*{fact*}{Fact}

\newtheorem{claim}{Claim}
\theoremstyle{definition}

\newcommand{\ee}{\varepsilon}
\newcommand{\nn}{\mathbb{N}}
\newcommand{\rr}{\mathbb{R}}

\begin{document}

\title[Subprojectivity]{Subprojectivity of projective tensor products \\ of Banach spaces of continuous functions}
\author{R.M. Causey}
\email{rmcausey1701@gmail.com}

\begin{abstract} Galego and Samuel showed that if $K,L$ are metrizable, compact, Hausdorff spaces, then $C(K)\widehat{\otimes}_\pi C(L)$ is $c_0$-saturated if and only if it is subprojective if and only if $K$ and $L$ are both scattered. We remove the hypothesis of metrizability from their result, and extend it from the case of the two-fold projective tensor product to the general $n$-fold projective tensor product to show that for any $n\in\mathbb{N}$ and compact, Hausdorff spaces $K_1, \ldots, K_n$, $\widehat{\otimes}_{\pi, i=1}^n C(K_i)$ is $c_0$-saturated if and only if it is subprojective if and only if each $K_i$ is scattered. 

\end{abstract}


\maketitle

\section{Introduction}

We recall that a Banach space $X$ is said to be \emph{subprojective} provided that any infinite dimensional subspace of $X$ admits a further infinite dimensional subspace which is complemented in $X$.   The notion of a subprojective space was introduced in \cite{W} to study preadjoints of strictly singular operators.   Subprojectivity was later shown to be closely related to perturbation classes (see \cite{OS} for more details).

As stated in \cite{GSP}, it is shown in \cite[Theorem $2.2$]{DF} that if a Banach space $X$ contains no isomorphic copy of $\ell_1$, then every isomorph of $c_0$ in $X$ admits a further subspace which is isomorphic to $c_0$ and complemented in $X$. Therefore every $c_0$-saturated space is subprojective.   Recall that a topological space $K$ is said to be \emph{scattered} (or \emph{dispersed}) if every non-empty subset of $K$ has an isolated point. We also recall that for a compact, Hausdorff space $K$, $K$ is scattered if and only if $C(K)$ is $c_0$-saturated, which is part of the main theorem of \cite{PS}. We also note that $C(K)$ is subprojective if and only if it is $c_0$-saturated. Indeed, one direction follows from the second sentence of the paragraph. If $K$ is not scattered, then again by the main theorem of \cite{PS}, $C(K)$ admits an isomorph of $\ell_1$, and this isomorph of $\ell_1$ can have no infinite-dimensional subspace which is complemented in $C(K)$. Indeed, if $P:C(K)\to X\subset C(K)$ is a projection onto $X$ and $X$ is a subspace of $\ell_1$, then $P$ is completely continuous, and therefore also weakly compact by a Theorem of Grothendieck \cite{Gr}.  Therefore $P^2$ is compact, and $X$ is finite dimensional. 

We recall that for Banach spaces $X_1, \ldots, X_n$, the $n$-fold projective tensor product $\widehat{\otimes}_{\pi,i=1}^n X_i$ is the completion of $\otimes_{i=1}^n X_i$ with respect to the norm \[\|u\|_\pi = \inf\Bigl\{\sum_{j=1}^m  \prod_{i=1}^n \|x_{ij}\|: m\in\nn, (x_{ij})_{i=1}^n\in \prod_{i=1}^n X_i\Bigr\}.\]     Of course, since subprojectivity passes to subspaces, if $K_1, \ldots, K_n$ are compact, Hausdorff spaces such that $\widehat{\otimes}_{\pi,i=1}^n C(K_i)$ is subprojective, then each $C(K_i)$ is subprojective and each $K_i$ is scattered.    The main result of this work is to prove the converse. The $n=1$ case of the following theorem  recovers the fact above that if $K$ is scattered, then $C(K)$ is  $c_0$-saturated, and our proof of the main theorem includes a self-contained proof of this special case. 

\begin{theorem} Fix $n\in\nn$ and  compact, Hausdorff spaces $K_1, \ldots, K_n$. The projective tensor product $\widehat{\otimes}_{\pi,i=1}^n C(K_i)$ is  $c_0$-saturated if and only if it is subprojective if and only if each space $K_i$ is scattered.  
\label{main1}
\end{theorem}

It was shown in \cite{GS} that if $K,L$ are metrizable, $C(K)\widehat{\otimes}_\pi C(L)$ is $c_0$-saturated, and therefore subprojective, if and only if $K,L$ are countable.   Theorem \ref{main1} generalizes their result in two directions: It extends the $n=2$ case to the general case, and it extends the result from metrizable spaces $K$ to arbitrary compact, Hausdorff spaces.  Moreover, separability, countability, and metrizability play no special role in our proof. Our proof makes explicit the Grasberg norm, which has previously only appeared implicitly. Previously a number of results known for $C(K)$ spaces with $K$ countable, compact, Hausdorff (such as those in \cite{S}, \cite{HLP}, and \cite{GS}) have been extended to the general case of scattered, compact, Hausdorff spaces (see \cite{C}, \cite{C2}), and the decomposition implicit in the Grasberg norm has been additionally used in, for example, \cite{C3} and \cite{CGS}.  Many of the aforementioned proofs in the case of $K$ countably infinite used the fact that $C(K)$ is isomorphic to $C(\omega^{\omega^\xi})$ for some countable $\xi$, which provides a convenient basis for induction. However, using this fact necessarily limits the results to metrizable $K$.  The Grasberg norm provides an appropriate structure to complete a proof by induction without this limitation. Moreover, these proofs demonstrate that the assumption that $K$ is an ordinal interval is an artifact rather than fundamental part of the proof, and that the Cantor-Bendixson index and derived set structure of $K$ are central to the results.

\section{The Grasberg norm}

Recall that for a topological space $K$ and a subset $L$ of $K$, $L'$ denotes the \emph{Cantor-Bendixson derivative} of $L$ (sometimes called the \emph{derived set}), which is defined to be the set of members of $L$ which are not isolated in $L$.    Note that if $L$ is a closed subset of $K$, then so is $L'$. In particular, if $K$ is compact and $L$ is closed, then $L'$ is also compact.     We define \[K^0=K,\] \[K^{\xi+1}=(K^\xi)',\] and for a limit ordinal $\xi$, \[K^\xi=\bigcap_{\zeta<\xi}K^\zeta.\] If $K$ is compact, then $K^\xi$ is compact for each ordinal $\xi$.    Note that $K$ is scattered if and only if there exists an ordinal $\xi$ such that $K^\xi=\varnothing$. If $K$ is scattered, we let $CB(K)$ denote the minimum ordinal $\xi$ such that $K^\xi=\varnothing$. The quantity $CB(K)$ is the \emph{Cantor-Bendixson index} of $K$.  If $K$ fails to be scattered, we use the notation $CB(K)=\infty$.     For a subset $S$ of an understood topological space $K$, we let $CB(S)$ denote the Cantor-Bendixson index of $S$ endowed with its subspace topology.  We let $CB(\varnothing)=0$. Note that for a compact, Hausdorff space $K$ satisfies $CB(K)=1$ if and only if $K$ is (non-empty and) finite. 

For a scattered, compact, Hausdorff space $K$, we let $\gamma(K)=0$ if $K$ is finite, and otherwise we let $\gamma(K)=\omega^\gamma$, where $\gamma$ is the unique ordinal such that $\omega^\gamma< CB(K) <\omega^{\gamma+1}$.   Let us explain how we know such an ordinal exists. First, we note that if we had instead defined $\gamma(K)$ to be the supremum of ordinals $\gamma$ such that $\omega^\gamma\leqslant CB(K)$, then by closedness of $[0, CB(K)]$, $\omega^\gamma \leqslant CB(K)$. We will explain why $\omega^\gamma<CB(K)$.  We outline this standard argument because similar arguments will be used throughout. If $\xi$ is a limit ordinal such that $\cap_{\zeta<\xi}K^\zeta=\varnothing$, then since the sets $(K^\zeta)_{\zeta<\xi}$ are compact, this collection must fail the finite intersection property.    Since this collection is a decreasing collection, it must be the case that $K^\zeta=\varnothing$ for some $\zeta<\xi$.  From this it follows that for compact, Hausdorff $K$, $CB(K)$ cannot be a limit ordinal.   Therefore if $\omega^\gamma\leqslant CB(K)$, then either $\gamma=0$ or $\gamma>0$. In the latter case,  $\omega^\gamma$ is a limit ordinal and the inequality $\omega^\gamma \leqslant CB(K)$ must be strict.  If $\gamma=0$, then $\omega^\gamma=\omega^0=1$. Since $K$ is infinite and compact, it must admit a point which fails to be isolated, which means $K^1\neq \varnothing$ and $\omega^\gamma=\omega^0=1<CB(K)$.

\begin{lemma} Let $K$ be a scattered, compact, Hausdorff topological space.    \begin{enumerate}[(i)]\item For closed subsets $S_0, \ldots, S_{k-1}$ of $K$, $\gamma\Bigl(\bigcup_{l=0}^{k-1}S_l\Bigr)=\max_{0\leqslant l<k} \gamma(S_l)$. \item For any ordinals $\zeta, \eta$, $(K^\eta)^\zeta=K^{\eta+\zeta}$. \item If $\xi$ is an ordinal and $S$ is a closed subset of $K^{\omega^\xi n}\setminus K^{\omega^\xi (n+1)}$ for some $n<\omega$, then $\gamma(S)<\omega^\xi$. \end{enumerate}
\label{ords}
\end{lemma}

\begin{proof} Of course $\gamma\Bigl(\bigcup_{l=0}^{k-1}S_l\Bigr)\geqslant \max_{0\leqslant l<k} \gamma(S_l)$.  Suppose that $\xi=\max_{0\leqslant l<k}\gamma(S_l)$. If $\xi=0$, this means each set $S_l$ is finite, and so is the union, so $\gamma\Bigl(\bigcup_{l=0}^{k-1}S_l\Bigr)=0=\xi$.  If $\xi>0$, then $\xi=\omega^\gamma$ for some ordinal $\gamma$ such that $S_l^{\omega^{\gamma+1}}=\varnothing$ for each $0\leqslant l<k$.  It is a standard property of the Cantor-Bendixson derivative that $\Bigl(\bigcup_{l=0}^{k-1}S_l\Bigr)^\zeta = \bigcup_{l=0}^{k-1}S_l^\zeta$ for each ordinal $\zeta$, so \[\Bigl(\bigcup_{l=0}^{k-1}S_l\Bigr)^{\omega^{\gamma+1}}\subset \bigcup_{l=0}^{k-1}S_l^{\omega^{\gamma+1}} = \varnothing,\] in which case $\gamma\Bigl(\bigcup_{l=0}^{k-1}S_l\Bigr)\leqslant\omega^\gamma  =\xi$. This yields $(i)$. 

Item  $(ii)$ is a  standard property of the  Cantor-Bendixson derivative.

Let us show $(iii)$.    If $\xi=0$, then $\omega^\xi=1$.  In this case, $S\subset K^n\setminus K^{n+1}$. This means that $S$ is finite, since if $S$ were infinite, it would contain a point which is not isolated in $S$, and therefore not isolated in $K^n\supset S$. This point would necessarily lie in $K^{n+1}\subset K\setminus S$. Since $S$ is finite, $\gamma(S)=0<1=\omega^\xi$.

   Next suppose that $\xi>0$.  In this case, $\omega^\xi$ is a limit ordinal.    Moreover, $(S^\zeta)_{\zeta<\omega^\xi}$ is a decreasing sequence of compact subsets of $K$.   This means that either $\cap_{\zeta<\omega^\xi} S^\zeta \neq \varnothing$ or the family $(S^\zeta)_{\zeta<\omega^\xi}$ fails the finite intersection property. The former cannot hold, because $S^\zeta \subset S\cap (K^{\omega^\xi n})^\zeta= K^{\omega^\xi n+\zeta}=\cap_{\mu\leqslant \omega^\xi n + \zeta}K^\mu$ for all $\zeta<\omega^\xi$.  Therefore \[\bigcap_{\zeta<\omega^\xi} S^\zeta \subset S \cap \bigcap_{\zeta<\omega^\xi}\bigcap_{\mu \leqslant \omega^\xi n + \zeta} K^\mu = S\cap \bigcap_{\mu<\omega^\xi(n+1)} K^\mu = S\cap K^{\omega^\xi(n+1)} = \varnothing.\]  This means $(S^\zeta)_{\zeta<\omega^\xi}$ fails the finite intersection property, so there exists a finite subset $F$ of $\omega^\xi$ such that $\cap_{\zeta\in F}S^\zeta =\varnothing$.  If $F=\varnothing$, this means $S=\varnothing$, and $\Gamma(S)=0<\omega^\xi$.  If $F\neq \varnothing$, then since the sets $(S^\zeta)_{\zeta<\omega^\xi}$ are decreasing, $S^{\max F}=\cap_{\zeta\in F}S^\zeta = \varnothing$. From this it follows that $CB(S)\leqslant \max F<\omega^\xi$, so $\Gamma(S)<\omega^\xi$.

\end{proof}

If $K$ is an infinite, scattered, compact, Hausdorff space, there exists a unique positive integer $k=k(K)$ such that $K^{{\gamma(K)}(k-1)}\neq \varnothing$ and $K^{{\gamma(K)}k}=\varnothing$.  This is because, by the definition of $\gamma$, $K^{{\gamma(K)}}\neq \varnothing$, and if $K^{\gamma(K) l}\neq \varnothing$ for all $l\in\nn$, then, again using the finite intersection property as in $(iii)$ above, this would imply that $K^{\gamma(K)\omega}=\cap_{l\in\nn} K^{\omega^{\gamma(K)}l}\neq \varnothing$.    But this is a contradiction of the definition of $\gamma(K)$, from which it follows that $\{l<\omega: K^{\gamma(K)l}\} $ is non-empty, and $k(K)=1+\max \{l<\omega: K^{\gamma(K)l}\}$.      In the sequel, $k(K)$ will denote this integer. When the space $K$ is understood and no confusion can arise, we will write $\gamma$ in place of $\gamma(K)$ and $k$ in place of $k(K)$.

If $K$ is infinite, scattered, compact, Hausdorff, then  \[K=K^{\gamma \cdot 0} \supsetneq K^{\gamma\cdot 1}\supsetneq \ldots \supsetneq K^{\gamma \cdot (k-1)} \supsetneq K^{\gamma\cdot k}=\varnothing.\]  Moreover, $K^{\gamma \cdot l}$ is compact for each $0\leqslant l< k$.    We define the \emph{Grasberg norm} $[\cdot]$ on $C(K)$ to be \[[f]= \max_{0\leqslant l <k} 2^l \|f|_{K^{\gamma \cdot l}}\|.\] Note that this norm is equivalent to the canonical norm on $C(K)$. 

If $K$ is finite, then we let $[\cdot]$ denote the canonical norm on $C(K)$ and refer to this as the Grasberg norm.  We isolate the following useful properties of the Grasberg norms.

\begin{proposition} Let $K$ be scattered, compact, Hausdorff. \begin{enumerate}[(i)]\item Suppose $K$ is infinite.   Let $\Phi\subset C(K)$ be norm compact and assume $c>\max_{f\in \Phi} [f]$.  Then if \[S=\bigcup_{l=0}^{k-1}\Bigl\{\varpi\in K^{\gamma\cdot l}:(\exists f\in \Phi)(2^{l+1}|f(\varpi)|\geqslant c)\Bigr\},\] it follows that  $\gamma(S)<\gamma(K)=:\gamma$. \item Suppose that the numbers $c>1$, $\ee>0$ and  the sets   $\Phi, \Upsilon\subset C(K)$ and  $S\subset K$ are such that \[S=\bigcup_{l=0}^{k-1} \Bigl\{\varpi\in K^{\gamma\cdot l}:(\exists f\in \Phi)(2^{l+1}|f(\varpi)|\geqslant c)\Bigr\},\] $[g]\leqslant c$ for each $g\in \Phi$, and for each $f\in \Upsilon$, $[f]\leqslant 1/2$ and $\|f|_S\|\leqslant \ee$.     Then \[\sup\Bigl\{[\delta g+\delta' f]: |\delta|=|\delta'|=1, (g,f)\in \Phi\times\Upsilon\Bigr\}<c+2^k \ee.\] Here we observe the convention that $\|f|_S\|=0$ if $S=\varnothing$. 

     \end{enumerate}

\label{grasberg}
\end{proposition}

\begin{proof} $(i)$   Assume $S$ is non-empty.   For $0\leqslant l<k$, define \[S_l=\Bigl\{\varpi\in K^{\gamma\cdot l}:(\exists f\in \Phi)(2^{l+1}|f(\varpi)|\geqslant c)\Bigr\}.\]  Note that each $S_l$ is closed, since \[S_l= K^{\gamma\cdot l} \cap \bigcup_{f\in \Phi} f^{-1}([c,\infty)),\] and $K^{\gamma\cdot l}$, $\bigcup_{f\in \Phi} f^{-1}([c,\infty))$ are closed. Closedness of $\bigcup_{f\in \Phi} f^{-1}([c,\infty))$ uses the fact that $\Phi$ is norm compact.  Since $S=\cup_{l=0}^{k-1}S_l$, it is sufficient to show that $\gamma(S_l)<\gamma$ for each $0\leqslant l<k$ by Lemma \ref{ords}$(i)$.   Since $S_l\subset K^{\gamma\cdot l}$ is closed, in order to show that $\gamma(S_l)<\gamma$, by Lemma \ref{ords}$(iii)$,  it is sufficient to show that $S_l\cap K^{\gamma\cdot (l+1)}=\varnothing$.   However, if $S_l\cap K^{\gamma\cdot (l+1)}\neq \varnothing$, then there would exist $f\in \Phi$ and $\varpi\in K^{\gamma\cdot (l+1)}$ such that $2^{l+1}|f(\varpi)|\geqslant c$.  Then \[[f] \geqslant 2^{l+1}|f(\varpi)| \geqslant c,\] contradicting the hypothesis that $\sup_{g\in \Phi}[g]<c$.

$(ii)$ By homogeneity, it is sufficient to prove that \[\sup\Bigl\{[ g+\delta f]: |\delta|=1, (g,f)\in \Phi\times\Upsilon\Bigr\}<c+2^k \ee.\]  Fix $\delta$ with $|\delta|=1$ and $(g,f)\in \Phi\times \Upsilon$.   Fix $0\leqslant l<k$ and $\varpi\in K^{\gamma\cdot l}$.    If $\varpi\in S$, then \begin{align*}2^l |g(\varpi)+\delta f(\varpi)| &  \leqslant [g]+2^l \|f|_S\| \leqslant c+2^{k-1}\ee. \end{align*}  If $\varpi\in K^{\gamma\cdot l}\setminus S$, then \[ 2^l |g(\varpi)| = 2^{l+1}|g(\varpi)|/2<c/2 \] and \begin{align*} 2^l|g(\varpi)+\delta f(\varpi)| & \leqslant 2^l|g(\varpi)|+2^l|f(\varpi)| \leqslant c/2+[f] \leqslant (c+1)/2. \end{align*}  Since $0\leqslant l<k$ and $\varpi\in K^{\gamma\cdot l}$ were arbitrary, \[[g+\delta f] \leqslant \max\{c+2^{k-1}\ee, (c+1)/2\}.\]  Since $\delta$ with $|\delta|=1$ and $(g,f)\in \Phi\times \Upsilon$ were arbitrary, \[\sup\Bigl\{[ g+\delta f]: |\delta|=1, (g,f)\in \Phi\times\Upsilon\Bigr\} \leqslant \max\{c+2^{k-1}\ee, (c+1)/2\} <c+2^k \ee.\]

\end{proof}

\section{Main results}

We recall that for a Banach space $X$ and a sequence $(x_j)_{j=1}^\infty\in \ell_\infty(X)$, \[\|(x_j)_{j=1}^\infty\|_1^w = \sup \Bigl\{\sum_{j=1}^\infty |\langle x^*, x_j\rangle| : x^*\in B_{X^*}\Bigr\} = \sup\Bigl\{\Bigl\|\sum_{j=1}^m \ee_j x_j\Bigr\|: m\in\nn, |\ee_j|=1\Bigr\}.\]  The sequence $(x_j)_{j=1}^\infty$ is said to be \emph{weakly unconditionally Cauchy} (or \emph{WUC}) if $\|(x_j)_{j=1}^\infty\|_1^w<\infty$.   Of course, a seminormalized basic sequence $(x_j)_{j=1}^\infty$ in $X$ is equivalent to the canonical $c_0$ basis if and only if $\|(x_j)_{j=1}^\infty\|_1^w<\infty$.

The following is an easy and standard fact about the tensorization of WUC sequences. 

\begin{proposition} For $n\in\nn$ and Banach spaces $X_1, \ldots, X_n$, suppose that for $1\leqslant i\leqslant n$, $(x_{ij})_{j=1}^\infty \in \ell_\infty(X_i)$ is given.  Then  \[\|(\otimes_{i=1}^n x_{ij})_{j=1}^\infty\|_1^w \leqslant \prod_{i=1}^n \|(x_{ij})_{j=1}^\infty\|_1^w.\] Here, $\otimes_{i=1}^n x_{ij}$ is considered as a member of $\widehat{\otimes}_{\pi,i=1}^n X_i$.   

\label{ck}
\end{proposition}

\begin{proof} We prove the result by induction on $n$. The base case $n=1$ is a tautology, since $\widehat{\otimes}_{\pi,i=1}^1 X_i=X_1$ by definition.    For the general inductive step, it is sufficient to prove the case $n=2$. Fix Banach spaces $X,Y$ and $(x_j)_{j=1}^\infty\in \ell_\infty(X)$, $(y_j)_{j=1}^\infty \in\ell_\infty(Y)$, $m\in\nn$, and $(\delta_j)_{j=1}^m$ with $|\delta_j|=1$ for each $1\leqslant j\leqslant m$.    Let $(\ee_j)_{j=1}^m$ be a Rademacher sequence. Note that \[\sum_{j=1}^m \delta_j x_j\otimes y_j = \mathbb{E} \Bigl[\Bigl(\sum_{j=1}^m \ee_j \delta_j x_j\Bigr)\otimes \Bigl(\sum_{j=1}^m \ee_j y_j\Bigr)\Bigr],\] so that \begin{align*} \Bigl\|\sum_{j=1}^m \delta_j x_j\otimes y_j \Bigr\|_\pi & \leqslant \mathbb{E}\Bigl\|\Bigl(\sum_{j=1}^m \ee_j \delta_j x_j\Bigr)\otimes \Bigl(\sum_{j=1}^m \ee_j y_j\Bigr)\Bigr\|_\pi \\ & = \mathbb{E}\Bigl\|\sum_{j=1}^n \ee_j\delta_j x_j \Bigr\|\Bigl\|\sum_{j=1}^m \ee_j y_j\Bigr\| \\ & \leqslant \mathbb{E} \|(x_j)_{j=1}^\infty\|_1^w \|(y_j)_{j=1}^\infty\|_1^w =  \|(x_j)_{j=1}^\infty\|_1^w \|(y_j)_{j=1}^\infty\|_1^w . \end{align*}

\end{proof}

\begin{corollary} For $n\in\nn$,  Banach spaces $X_1, \ldots, X_n$, and a constant $C>0$, suppose that for each $1\leqslant i\leqslant n$ and  $j\in\nn$, $\Upsilon_{ij}\subset B_{X_i}$ is given such that \[\sup\Bigl\{\|(x_j)_{j=1}^\infty\|_1^w: 1\leqslant i\leqslant n, (x_j)_{j=1}^\infty \in \prod_{j=1}^\infty \Upsilon_{ij}\Bigr\} \leqslant C.\]  Suppose also that two sequences $(u_j)_{j=1}^\infty, (v_j)_{j=1}^\infty \subset \widehat{\otimes}_{\pi,i=1}^n X_i$  are given such that $(u_j)_{j=1}^\infty$ is a seminormalized basic sequence, $\sum_{j=1}^\infty\|u_j-v_j\|<\infty$, and \[ (v_j)_{j=1}^\infty \in \prod_{j=1}^\infty \text{\emph{co}}\Bigl\{\otimes_{i=1}^n x_i: (x_i)_{i=1}^n \in \prod_{i=1}^n \Upsilon_{ij}\Bigr\}.\]  Then $(u_j)_{j=1}^\infty$ is equivalent to the canonical $c_0$ basis. 
\label{dl}
\end{corollary}

\begin{proof} Since $(u_j)_{j=1}^\infty$ is seminormalized and basic, it is sufficient to show that $\|(u_j)_{j=1}^\infty\|_1^w<\infty$.  Since \[\|(u_j)_{j=1}^\infty\|_1^w \leqslant \|(v_j)_{j=1}^\infty\|_1^w + \sum_{j=1}^\infty \|v_j-u_j\|\] and $\sum_{j=1}^\infty \|v_j-u_j\|<\infty$, it is sufficient to show that $\|(v_j)_{j=1}^\infty\|_1^w<\infty$.   Note that since $\|(x_j)_{j=1}^\infty \|_1^w \leqslant C$ for each $1\leqslant i\leqslant n$ and $(x_j)_{j=1}^\infty\in \prod_{j=1}^\infty \Upsilon_{ij}$, it follows from Proposition \ref{ck} that \[\|(\otimes_{i=1}^n x_{ij})_{j=1}^\infty\|_1^w \leqslant C^n \] for any sequences $(x_{ij})_{j=1}^\infty$, $1\leqslant i\leqslant n$,  such that $(x_{ij})_{j=1}^\infty\in \prod_{j=1}^\infty \Upsilon_{ij}$ for each $1\leqslant i\leqslant n$.       From this it follows that for any $m\in\nn$ and scalars $(\delta_j)_{j=1}^m$ such that $|\delta_j|\leqslant 1$ for each $1\leqslant j\leqslant m$, \[\Bigl\|\sum_{j=1}^m \delta_j \otimes_{i=1}^n x_{ij}\Bigr\|_\pi \leqslant C^n\] for any sequences $(x_{ij})_{j=1}^\infty$, $1\leqslant i\leqslant n$,  such that $(x_{ij})_{j=1}^\infty\in \prod_{j=1}^\infty \Upsilon_{ij}$ for each $1\leqslant i\leqslant n$.   Therefore \begin{align*} \sum_{j=1}^m \delta_j v_j & \in \text{co}\Bigl\{\sum_{j=1}^m \delta_j \otimes_{i=1}^n x_{ij}: (x_{ij})_{j=1}^\infty\in \prod_{j=1}^\infty \Upsilon_{ij}\Bigr\} \subset C^n B_{\widehat{\otimes}_{\pi,i=1}^n X_i}\end{align*}  for any such $m\in\nn$ and $(\delta_j)_{j=1}^m$.   Since this holds for any  $m\in\nn$ and $(\delta_j)_{j=1}^m\in B_{\ell_\infty^m}$, $\|(v_j)_{j=1}^\infty\|_1^w \leqslant C^n$.

\end{proof}

The next result was shown in \cite{CD}.    

\begin{theorem} For every $n\in\nn$ and scattered, compact, Hausdorff spaces $K_1, \ldots, K_n$, \[\Bigl(\widehat{\otimes}_{\pi,i=1}^n C(K_i)\Bigr)^* = \widehat{\otimes}_{\ee,i=1}^n \ell_1(K_i).\] 
\label{jl}
\end{theorem}

\begin{lemma} Fix $n\in\nn$. Let $K_1, \ldots, K_n$ be scattered, compact, Hausdorff spaces and let $S_1, \ldots, S_n$ be such that $S_i$ is a closed subset of $K_i$ for each $1\leqslant i\leqslant n$.  For $1\leqslant m\leqslant n$, if $S_m \neq \varnothing$, define \[Q_m:\widehat{\otimes}_{\pi,i=1}^n C(K_i)\to \widehat{\otimes}_{\pi,i=1}^n C(L_{i,m})\]  to be the extension of the map \[Q_m\otimes_{i=1}^n f_i = \otimes_{i=1}^n g_i,\] where \[L_{i,m} = \left\{ \begin{array}{ll} K_i & i\neq m \\ S_m & i = m\end{array}\right.\] and \[g_i = \left\{ \begin{array}{ll} f_i & i\neq m \\ f_m|_{S_m} & i = m.\end{array}\right.\]  If $S_m=\varnothing$, let $Q_m:\widehat{\otimes}_{\pi,i=1}^n C(K_i)\to \{0\}$ be the zero operator.  Suppose that $w\in B_{\widehat{\otimes}_{\pi,i=1}^n C(K_i)}$ and $\ee>0$ are such that $\|Q_m w\|<\ee/2^n$ for each $1\leqslant m\leqslant n$.    Then there exists \[v\in \text{\emph{co}}\{\otimes_{i=1}^n f_i: f_i\in B_{C(K_i)}, f_i|_{S_i}\equiv 0\} \] such that $\|w-v\|<\ee$.  Here, if $S_i=\varnothing$, we agree to the convention that $f|_{S_i}\equiv 0$ for each $f\in C(K_i)$. 
\label{nm}
\end{lemma}

{\color{black} 

\begin{proof}  Let \[C= \text{co}\{\otimes_{i=1}^n f_i: f_i\in B_{C(K_i)}, f_i|_{S_i}\equiv 0\} \] and let \[U=C+\{w\in \widehat{\otimes}_{\pi,i=1}^n C(K_i): \|w\|<\ee\}.\]

Suppose the result is not true.  If no such $v$ exists, then $w\notin U$.   Note that $U$ is open and convex, so by the Hahn-Banach theorem, there exist $\nu\in S_{(\widehat{\otimes}_{\pi,i=1}^n C(K_i))^*} = S_{\widehat{\otimes}_{\ee,i=1}^n \mathcal{M}(K_i)}$ and $\alpha\in \rr$ such that \[\sup_{u\in U} \text{Re\ }\langle \nu, u\rangle <\alpha \leqslant \text{Re\ }\langle \nu, w\rangle.\]  

Note that \[\ee+\sup_{u\in C}\text{Re\ }\langle \nu,u\rangle = \sup_{u\in U}\text{Re\ }\langle \nu,u\rangle \leqslant \text{Re\ }\langle \nu, w\rangle.\]

For $1\leqslant j\leqslant n$, define $T^0_j, T^1_j:\mathcal{M}(K_j)\to \mathcal{M}(K_j)$ by letting \[T^0_j\mu(L)=\mu(L\cap S_j)\] and \[T^1_j\mu(L)= \mu(L\setminus S_j).\]   Note that $\|T^0_j\|, \|T^1_j\| \leqslant 1$. For each $\eta\in\{0,1\}^n$, define \[T^\eta=\otimes_{j=1}^n T^{\eta(j)}_j:\widehat{\otimes}_{\ee,i=1}^n \mathcal{M}(K_j)\to \widehat{\otimes}_{\ee,i=1}^n \mathcal{M}(K_j).\]    Note that \[\|T^\eta\|= \prod_{j=1}^n \|T^{\eta(j)}_j\|\leqslant 1.\]  Therefore for each $\eta\in \{0,1\}^n$, $\|T^\eta\nu\|\leqslant 1$. Note also that $\nu = \sum_{\eta\in \{0,1\}^n} T^\eta \nu$.

Let $1_n\in \{0,1\}^n$ be the sequence such that $1_n(j)=1$ for each $j\in \{1, \ldots, n\}$.  Note that \[\sup_{u\in C} \text{Re\ }\langle \nu, u\rangle = \|T^{1_n}\nu\|.\]

For $\eta\in \{0,1\}^n$, we let $|\eta|=|\{j\in \{1, \ldots, n\}: \eta(j)=1\}|$.

\begin{claim} For $\eta\in \{0,1\}^n$, if $|\eta|<n$, then $|\langle T^\eta, w\rangle|<\ee/2^n$. \end{claim}

\begin{proof}[Proof of Claim $1$] If $|\eta|<n$, then there exists $1\leqslant i\leqslant n$ such that $\eta(i)=0$.  If $S_i=\varnothing$, then $T_i^0$ is the zero operator, and so is $T^\eta$. In this case, the inequality $|\langle T^\eta, u_0\rangle|<\ee/2^n$ is trivial. Now assume that $S_i\neq\varnothing$. Define $P_i:\mathcal{M}(K_i)\to \mathcal{M}(S_i)$ by letting $P_i\mu(L)=\mu(L)$.    Note that for any $f\in C(K_i)$ and $\mu\in \mathcal{M}(K_i)$, \[\langle T^0_i \mu, f\rangle = \langle P_iT^0_i \mu, f|_{S_i}\rangle = \langle P_i\mu, f|_{S_i}\rangle.\]  Note also that $\|P_i\|=1$.   

Define \[R_j=\left\{\begin{array}{ll} T^{\eta(j)}_j & : j \neq i \\ P_i & : j = i, \end{array} \right.\] and  \[R=\otimes_{j=1}^n R_j:\widehat{\otimes}_{\ee,j=1}^n \mathcal(K_j) \to \mathcal{M}(K_1)\widehat{\otimes}_\ee \ldots \widehat{\otimes}_\ee \mathcal{M}(K_{i-1}) \widehat{\otimes}_\ee \mathcal{M}(S_i)\widehat{\otimes}_\ee \mathcal{M}(K_{i+1}) \widehat{\otimes}_\ee \ldots \widehat{\otimes}_\ee \mathcal{M}(K_n),\] and note that $\|R\|\leqslant 1$.

Fix $\nu_1=\sum_{m=1}^M a_m \otimes_{j=1}^n \mu_{m,j}\in \widehat{\otimes}_{\ee,j=1}^n \mathcal{M}(K_j)$ and $u_1=\sum_{l=1}^L b_l\otimes_{j=1}^n f_{l,j}\in \widehat{\otimes}_{\pi,j=1}^n C(K_j)$ arbitrary.   Then \begin{align*} |\langle T^\eta \nu_1, u_1\rangle| & = \Bigl|\sum_{m=1}^M\sum_{l=1}^L a_mb_l \prod_{j=1}^n \langle T^{\eta(j)}_j\mu_{m,j}, f_{l,j}\rangle \Bigr| \\ & = \Bigl|\sum_{m=1}^M\sum_{l=1}^L a_mb_l \langle T^0_k \mu_{m,i}, f_{l,i}\rangle  \prod_{i\neq j=1}^n \langle T^{\eta(j)}_j\mu_{m,j}, f_{l,j}\rangle \Bigr| \\ & = \Bigl|\sum_{m=1}^M\sum_{l=1}^L a_mb_l \langle  P_i\mu_{m,i}, f_{l,i}|_{S_i}\rangle  \prod_{i\neq j=1}^n \langle T^{\eta(j)}_j\mu_{m,j}, f_{l,j}\rangle \Bigr| \\ & =|\langle R\nu_1, Q_i u_1\rangle| \leqslant \|R\nu_1\|\|Q_iu_1\| \leqslant \|\nu_1\|\|Q_iu_1\|.\end{align*}

Since linear combinations of elementary tensors are dense in both $\widehat{\otimes}_{\pi,j=1}^n C(K_j)$ and $\widehat{\otimes}_{\ee,j=1}^n \mathcal{M}(K_j)$, it follows from the previous computation that $|\langle T^\eta \nu,w\rangle|\leqslant \|\nu\|\|Q_iw\|=\|Q_iw\|<\ee/2^n$. This gives the claim.

\end{proof}

By the claim, it follows that \begin{align*} \ee+\|T^{1_n}\nu\| & = \ee+\sup_{u\in C}\text{Re\ }\langle \nu, u\rangle \leqslant  \text{Re\ }\langle \nu, w\rangle  \leqslant |\langle \nu, w\rangle| \leqslant |\langle T^{1_n}\nu, w\rangle|+\sum_{\underset{|\eta|<n}{\eta\in \{0,1\}^n}} |\langle T^\eta \nu, w\rangle| \\ & \leqslant \|T^{1_n}\nu\|+\ee \cdot \frac{2^n-1}{2^n}.\end{align*}  This is a contradiction, and this contradiction finishes the proof.

\end{proof}

}

The following is an easy sufficient condition for being $c_0$-saturated.

\begin{proposition} Let $X$ be a Banach space and let $\mathfrak{G}$ be a collection of operators from $X$ into (possibly different) Banach spaces.  Assume that \begin{enumerate}[(a)]\item for each $Q:X\to Y_Q\in \mathfrak{G}$, $Y_Q$ is $c_0$-saturated, and \item for any infinite dimensional subspace $Z$ of $X$, if $Q|_Z:Z\to Y_Q$ is strictly singular for each $Q\in \mathfrak{G}$, then $Z$ contains an isomorph of $c_0$. \end{enumerate} Then $X$ is $c_0$-saturated.   
\label{wr}
\end{proposition}

\begin{proof} Let $Z$ be an infinite dimensional subspace of $X$. Then either there exist $Q:X\to Y_Q\in \mathfrak{G}$ and an infinite dimensional subspace $W$ of $Z$ such that $Q|_W$ is an isomorphism, in which case $W$ is isomorphic to a subspace of the $c_0$-saturated space $Y_Q$ and therefore has a further subspace isomorphic to $c_0$, or $Q|_Z:Z\to Y_Q$ is strictly singular for each $Q\in \mathfrak{G}$, in which case $Z$ contains an isomorph of $c_0$ by $(b)$.   

\end{proof}

We next show how the preceding proposition will be combined with the Grasberg norm. 

\begin{proposition} For $n\in\nn$, let $K_1, \ldots, K_n$ be infinite, scattered, compact, Hausdorff spaces.  For $1\leqslant m\leqslant n$ and a closed, non-empty subset $S$ of $K_m$, let $Q^m_S:\widehat{\otimes}_{\pi,i=1}^n C(K_i)\to Y_{Q^m_S}$ be the extension of  \[Q^m_S\otimes_{i=1}^n f_i=\otimes_{i=1}^n g_i,\] where \[g_i = \left\{\begin{array}{ll} f_i & i \neq m \\ f_m|_S & i = m\end{array}\right.\]  and  \[Y_{Q^m_S} = C(K_1)\widehat{\otimes}_\pi \ldots \widehat{\otimes}_\pi C(K_{m-1})\widehat{\otimes}_\pi C(S) \widehat{\otimes}_\pi C(K_{m+1}) \widehat{\otimes}_\pi \ldots \widehat{\otimes}_\pi C(K_n).\]  For $1\leqslant m\leqslant n$, we let $Y_{Q^m_\varnothing}$ be the zero vector space and let $Q^m_\varnothing:\widehat{\otimes}_{\pi,i=1}^n C(K_i)\to Y_{Q^m_\varnothing}$ be the zero operator.    Let \[\mathfrak{G}_m=\{Q^m_S:\widehat{\otimes}_{\pi,i=1}^n C(K_i)\to Y_{Q^m_s}: S\subset K_m, S\text{\ closed}, \gamma(S)<\gamma(K_m).\}\] and let $\mathfrak{G}=\cup_{m=1}^n \mathfrak{G}_m$.    If $Z$ is any infinite dimensional subspace $Z$ of $\widehat{\otimes}_{\pi,i=1}^n C(K_i)$ such that $Q|_Z:Z\to Y_Q$ is strictly singular for all $Q\in \mathfrak{G}$, then $Z$ contains an isomorph of $c_0$.   
\label{st}
\end{proposition}

\begin{proof}     

Recall that for any Banach spaces $E,F$, any finite collection $A_1, \ldots, A_t:E\to F$ of strictly singular operators, any infinite dimensional subspace $G$ of $E$, and any $\ee>0$, there exists $e\in S_E$ such that $\|A_j e\|<\ee$ for each $1\leqslant j\leqslant t$.  Fix $\ee>0$ and numbers $(\ee_j)_{j=1}^\infty\subset (0,1)$ such that $\prod_{j=1}^\infty (1+\ee_j)<1+\ee$. Let $k=\max_{1\leqslant i\leqslant n}k(K_i)$.    Let $\gamma_i=\gamma(K_i)$ for each $1\leqslant i\leqslant n$.

Assume that $Z$ is an infinite dimensional subspace of $\widehat{\otimes}_{\pi,i=1}^n C(K_i)$ such that $Q|_Z:Z\to Y_Q$ is strictly singular for all $Q\in \mathfrak{G}$.    Fix $e_1\in S_Z$.    There exist finite subsets $\Lambda_{1,1}, \ldots, \Lambda_{n,1}$ of $B_{C(K_1)}, \ldots, B_{C(K_n)}$, respectively, and \[w_1\in \text{co}\Bigl\{\otimes_{i=1}^n f_i: (f_i)_{i=1}^n\in \prod_{i=1}^n\Lambda_{i,1}\Bigr\}\] such that $\|e_1-w_1\|<\ee_1$.     For $1\leqslant i\leqslant n$, let \[\Phi_{i,1}=\Upsilon_{i,1}= \{2^{-k}f: f\in \Lambda_{i,1}\}\] and note that $[f]\leqslant 1/2$ for all $f\in \Upsilon_{i,1}$.     Note also that \[v_1:= w_1/2^{nk}\in \text{co}\{\otimes_{i=1}^n f_i: (f_i)_{i=1}^n\in \prod_{i=1}^n \Upsilon_{i,1}\}.\]   Let $c_1=1$ and note that for each $1\leqslant i\leqslant n$, \[\sup_{f\in \Phi_{i,1}} [f] \leqslant 1/2<c_1=1.\]    For each $1\leqslant i\leqslant n$, let \[S_{i,1}=\bigcup_{l=0}^{k_i-1}\Bigl\{\varpi\in K^{\gamma_i \cdot l}_i: (\exists f\in \Phi_{i,1})(2^{l+1}|f(\varpi)|\geqslant c_1\Bigr\}\] and note that $\gamma(S_{i,1})<\gamma(K_i)$ by Proposition \ref{grasberg}$(i)$.      Let $u_1=e_1/2^{nk}$.  Choose a finite subset $G_1\subset B_{(\widehat{\otimes}_{\pi,i=1}^n C(K_i))^*}$ such that \[\inf_{e\in S_Z\cap \text{span}\{e_1\}} \max_{e^*\in G_1} |\langle e^*, e\rangle| \geqslant \frac{1}{1+\ee_1}.\]    This completes the base step of our recursion.

Assume that for some $m\in\nn$, we have \begin{enumerate}[(a)]\item $e_1, \ldots, e_m, u_1, \ldots, u_m\in Z$ such that $u_j=e_j/2^{nk}$ for each $1\leqslant j\leqslant m$, \item finite sets $ \Upsilon_{i,j}\subset C(K_i)$, $1\leqslant i\leqslant n$, $1\leqslant j \leqslant m$,  and compact sets $\Phi_{i,j}\subset C(K_i)$, $1\leqslant i\leqslant n,$ $1\leqslant j\leqslant m$,  such that \[\Phi_{i,j} = \Bigl\{\sum_{l=1}^m \delta_l f_l: |\delta_l|=1, (f_l)_{l=1}^m\in \prod_{l=1}^j \Upsilon_{i,l}\Bigr\},\] $\sup_{f\in \Upsilon_{i,j}} [f]\leqslant 1/2$, and $\sup_{g\in \Phi_{i,j}}[g] < 1+2^k\sum_{l=1}^{j-1} \ee_l$, \item $ w_1, \ldots, w_m, v_1, \ldots, v_m \in \widehat{\otimes}_{\pi,i=1}^n C(K_i)$ such that $v_j=w_j/2^{nk}$ and \[v_j \in \text{co}\Bigl\{\otimes_{i=1}^n f_i: (f_i)_{i=1}^n\in \prod_{i=1}^n \Upsilon_{i,j}\Bigr\}\] for each $1\leqslant j\leqslant m$, \item closed sets $S_{i,j}\subset K_i$, $1\leqslant i\leqslant n$, $1\leqslant j\leqslant m$,   such that \[S_{i,j}=\bigcup_{l=0}^{k_i-1}\Bigl\{\varpi\in K_i^{\gamma_i\cdot l}: (\exists f\in \Phi_{i,j})(2^{l+1}|f(\varpi)|\geqslant c_j)\Bigr\},\] where $c_j=1+2^k\sum_{l=1}^{j-1}\ee_l$,  \item finite  sets $G_j \subset B_{(\widehat{\otimes}_{\pi,i=1}^n C(K_i))^*}$, $1\leqslant j\leqslant m$,  such that \[\inf_{e\in S_Z\cap  \text{span}\{e_1, \ldots, e_j\}} \max_{e^*\in G_j} |\langle e^*, e\rangle| \geqslant \frac{1}{1+\ee_j}\] for each $1\leqslant j \leqslant m$.      \end{enumerate}

Note that by $(b)$ and $(d)$, combined with Proposition \ref{grasberg}$(i)$, $\gamma(S_{i,m})<\gamma_i$ for each $1\leqslant i\leqslant n$.    By hypothesis, for each $1\leqslant i\leqslant n$, $Q^i_{S_{i,m}}|_Z$ is strictly singular (we recall that  $Q^i_{S_{i,m}}$ is taken to be the zero operator into the zero vector space if $S_{i,m}=\varnothing$).   Note that \[Z\cap \bigcap_{e^*\in G_m}\ker(e^*)\] is an infinite dimensional subspace of $Z$. By the remark at the beginning of the proof, there exists $e_{m+1}\in S_Z\cap \bigcap_{e^*\in G_m}\ker(e^*)$ such that $\|Q^i_{S_{i,m}}e_{m+1}\|< \ee_m/2^n$ for each $1\leqslant i\leqslant n$.   By Lemma \ref{nm}, there exists \[w_{m+1}\in \text{co}\Bigl\{\otimes_{i=1}^n f_i: f_i\in B_{C(K_i)}, f_i|_{S_{i,m}}\equiv 0\Bigr\}\] such that $\|e_{m+1}-w_{m+1}\|\leqslant \ee_m$.   Define $u_{m+1}=e_{m+1}/2^{nk}$, $v_{m+1}=w_{m+1}/2^{nk}$.      For each $1\leqslant i\leqslant n$, let \[\Lambda_{i, m+1} \subset \{f\in B_{C(K_i)}:f|_{S_{i,m}}\equiv 0\}\] be a finite set such that \[w_{m+1}\in \text{co}\Bigl\{\otimes_{i=1}^n f_i: (f_i)_{i=1}^n \in \prod_{i=1}^n \Lambda_{i ,m+1}\Bigr\}.\]   Let \[\Upsilon_{i ,m+1}= \{f/2^k: f\in \Lambda_{i, m+1}\}\] and note that $[f]\leqslant 1/2$ for all $f\in \Upsilon_{i, m+1}$ and \[v_{m+1}\in \text{co}\Bigl\{\otimes_{i=1}^n f_i: (f_i)_{i=1}^n\in \prod_{i=1}^n \Upsilon_{i, m+1}\Bigr\}.\]    For $1\leqslant i\leqslant n$, define \[\Phi_{i,m+1}=\Bigl\{\sum_{j=1}^{m+1} \delta_j f_j: |\delta_j|=1, (f_j)_{j=1}^{m+1}\in \prod_{j=1}^{m+1}\Upsilon_{i,m+1}\Bigr\}.\] Since $f|_{S_{i,m}}\equiv 0$ for each $1\leqslant i\leqslant n$ and $f\in \Upsilon_{i,m+1}$, and since $(b)$ holds, it follows from Proposition \ref{grasberg}$(ii)$ that $\sup_{f\in \Phi_{i,m+1}}[f]<c_{m+1}:=1+2^k\sum_{j=1}^m \ee_j$.     Define \[S_{i,m+1}=\bigcup_{l=0}^{k_i-1}\Bigl\{\varpi\in K_i^{\gamma_i\cdot l}: (\exists f\in \Phi_{i,m+1})(2^{l+1}|f(\varpi)|\geqslant c_{m+1})\Bigr\}\] and choose a finite set $G_{m+1}\subset B_{(\widehat{\otimes}_{\pi,i=1}^n C(K_i))^*}$ such that \[\inf_{e\in S_Z\cap \text{span}\{e_1, \ldots, e_{m+1}\}} \max_{e^*\in G_{m+1}}|\langle e^*, e\rangle|\geqslant \frac{1}{1+\ee_{m+1}}.\]    This completes the recursive construction.   

We claim that $(e_j)_{j=1}^\infty$ is basic with basis constant not more than $1+\ee$, and that \[\|(e_j)_{j=1}^\infty\|_1^w <\infty.\]  Since $(e_j)_{j=1}^\infty$ is normalized, this will finish the proof.  To see the first inequality, note that since $e_{m+1}\in \cap_{e^*\in G_m}\ker(e^*)$, \begin{align*} (1+\ee_m)\Bigl\|\sum_{i=1}^{m+1} a_je_j\Bigr\| & \geqslant (1+\ee_m) \max_{e^*\in G_m}\Bigl|\Bigl\langle e^*, \sum_{j=1}^m a_je_j\Bigr\rangle\Bigr|   \geqslant \Bigl\|\sum_{j=1}^m a_je_j\Bigr\|\end{align*} for any scalars $(a_j)_{j=1}^{m+1}$.  Therefore for any $(a_j)_{j=1}^\infty\in c_{00}$ and $m<l$, \begin{align*} \Bigl\|\sum_{j=1}^l a_je_j\Bigr\| & \leqslant (1+\ee_{l-1})\Bigl\|\sum_{j=1}^{l-1} a_je_j\Bigr\|\leqslant (1+\ee_{l-1})(1+\ee_{l-2})\Bigl\|\sum_{j=1}^{l-2}a_je_j\Bigr\| \\ & \leqslant \ldots \leqslant (1+\ee_{l-1})(1+\ee_{l-2})\ldots (1+\ee_m)\Bigl\|\sum_{j=1}^m a_je_j\Bigr\| \leqslant (1+\ee)\Bigl\|\sum_{j=1}^m a_je_j\Bigr\|. \end{align*} This implies that $(u_j)_{j=1}^\infty$ is basic with basis constant not more than $1+\ee$.

For the remaining inequality, we first note that by $(b)$ of the recursive inequality, it follows that for any $1\leqslant i\leqslant n$ and $m\in\nn$,  \[\sup_{f\in \Phi_{i,m}}[f] \leqslant 1+2^k\sum_{j=1}^\infty \ee_j.\]     This is equivalent to the condition that for each $1\leqslant i\leqslant n$ and $(f_j)_{j=1}^\infty\in \prod_{j=1}^\infty \Upsilon_{i,j}$, $\|(f_j)_{j=1}^\infty\|_1^w \leqslant 1+2^k\sum_{j=1}^\infty \ee_j$.    It follows from Corollary \ref{dl} that $(u_j)_{j=1}^\infty $ is equivalent to the canonical $c_0$ basis. Since $(e_j)_{j=1}^\infty$ is a constant multiple of the sequence $(u_j)_{j=1}^\infty$, $(e_j)_{j=1}^\infty$ is equivalent to the canonical $c_0$ basis. .

\end{proof}

\begin{proof}[Proof of Theorem \ref{main1}]

By convention, the $0$-fold projective tensor product of the empty sequence of Banach spaces will be the scalar field. 

Let \[\Gamma=\{0\}\cup \{\omega^\gamma: \gamma\in\textbf{Ord}\}.\]    For $n<\omega$, let $\mathcal{P}_n$ be the proposition that for any scattered, compact, Hausdorff spaces $K_1, \ldots, K_n$, $\widehat{\otimes}_{\pi,i=1}^n C(K_i)$ is $c_0$-saturated.  We prove this by induction on $n$.   The $n=0$ case follows from the convention mentioned in the first sentence of the proof. That is, for $n=0$, the projective tensor product in question is simply the scalar field, which is vacuously $c_0$-saturated. 

Assume that for $0<n<\omega$, $\mathcal{P}_{n-1}$ holds. For $(\xi_1, \ldots, \xi_n)\in \Gamma^n$, let $\mathcal{P}_n(\xi_1, \ldots, \xi_n)$ be the proposition that for any scattered, compact, Hausdorff spaces $K_1, \ldots, K_n$ such that $\gamma(K_i)\leqslant \xi_i$ for each $1\leqslant i\leqslant n$, $\widehat{\otimes}_{\pi,i=1}^n C(K_i)$ is $c_0$-saturated.  We prove $\mathcal{P}_n(\xi_1, \ldots, \xi_n)$ by induction for $(\xi_1, \ldots, \xi_n)\in \Gamma^n$, ordered lexicographically.   Fix $(\xi_1, \ldots, \xi_n)\in \Gamma^n$ and assume $\mathcal{P}_n(\zeta_1, \ldots, \zeta_n)$ holds for each $(\zeta_i)_{i=1}^n\in \Gamma^n$ with $(\zeta_i)_{i=1}^n <_\text{lex} (\xi_i)_{i=1}^n$. Let $K_1, \ldots, K_n$ be scattered, compact, Hausdorff spaces such that $\gamma(K_i)\leqslant \xi_i$ for each $1\leqslant i\leqslant n$. Consider two cases. 

Case $1$: For some $1\leqslant m\leqslant n$, $K_m$ is finite.        Then, with $\approx$ denoting isomorphism, \begin{align*} \widehat{\otimes}_{\pi,i=1}^n C(K_i) & \approx C(K_m)\widehat{\otimes}_\pi \Bigl[C(K_1)\widehat{\otimes}_\pi \ldots \widehat{\otimes}_\pi C(K_{m-1})\widehat{\otimes}_\pi C(K_{m+1}) \widehat{\otimes}_\pi \ldots \widehat{\otimes}_\pi C(K_n)\Bigr] \\ & \approx \ell_\infty^{|K_m|}\widehat{\otimes}_\pi \Bigl[C(K_1)\widehat{\otimes}_\pi \ldots \widehat{\otimes}_\pi C(K_{m-1})\widehat{\otimes}_\pi C(K_{m+1}) \widehat{\otimes}_\pi \ldots \widehat{\otimes}_\pi C(K_n)\Bigr] \\ & \approx \ell_\infty^{|K_m|}\Bigl(C(K_1)\widehat{\otimes}_\pi \ldots \widehat{\otimes}_\pi C(K_{m-1})\widehat{\otimes}_\pi C(K_{m+1}) \widehat{\otimes}_\pi \ldots \widehat{\otimes}_\pi C(K_n)\Bigr).\end{align*} Since $\mathcal{P}_{n-1}$ holds, $C(K_1)\widehat{\otimes}_\pi \ldots \widehat{\otimes}_\pi C(K_{m-1})\widehat{\otimes}_\pi C(K_{m+1}) \widehat{\otimes}_\pi \ldots \widehat{\otimes}_\pi C(K_n)$ is $c_0$-saturated.  Note that in the case $n=1$,  $C(K_1)\widehat{\otimes}_\pi \ldots \widehat{\otimes}_\pi C(K_{m-1})\widehat{\otimes}_\pi C(K_{m+1}) \widehat{\otimes}_\pi \ldots \widehat{\otimes}_\pi C(K_n)=\mathbb{K}$ by the convention established at the beginning of the proof.     Since  \[\widehat{\otimes}_{\pi,i=1}^n C(K_i)\approx \ell_\infty^{|K_m|}\Bigl(C(K_1)\widehat{\otimes}_\pi \ldots \widehat{\otimes}_\pi C(K_{m-1})\widehat{\otimes}_\pi C(K_{m+1}) \widehat{\otimes}_\pi \ldots \widehat{\otimes}_\pi C(K_n)\Bigr),\] and the latter space is a finite direct sum of $c_0$-saturated spaces, it is $c_0$-saturated.    Since being $c_0$-saturated is an isomorphic invariant, $\widehat{\otimes}_{\pi,i=1}^n C(K_i)$ is $c_0$-saturated.

Case $2$: For each $1\leqslant i\leqslant n$, $K_i$ is infinite.   In this case, Proposition \ref{st} applies.  Let $\mathfrak{G}$ be as defined in Proposition \ref{st}.   We claim that $\mathfrak{G}$ satisfies $(a)$ and $(b)$ of Proposition \ref{wr}. It is the content of Proposition \ref{st} that $\mathfrak{G}$ satisfies $(b)$ of Proposition \ref{wr}. For $Q\in \mathfrak{G}$, $Q=Q^m_S$ for some $1\leqslant m\leqslant S$ and $S\subset K_m$ closed with $\gamma(S)<\gamma(K_m)$.  In this case, either $Y_Q=\{0\}$ if $S=\varnothing$ or $Y_Q=\widehat{\otimes}_{\pi,i=1}^n C(L_i)$, where \[L_i= \left\{\begin{array}{ll} K_i & : i\neq m \\ S & : i = m.\end{array}\right.\]  If $S=\varnothing$, then $Y_Q$ is vacuously $c_0$-saturated. If $S\neq \varnothing$, then $(\gamma(L_i))_{i=1}^n <_\text{lex} (\xi_i)_{i=1}^n$.  By the inner inductive hypothesis, $\mathcal{P}_n(\gamma(L_1), \ldots, \gamma(L_n))$ holds, so $Y_Q$ is $c_0$-saturated. Since $Q\in \mathfrak{G}$ was arbitrary, $\mathfrak{G}$ satisfies $(a)$ and $(b)$ of Proposition \ref{wr}.   By the conclusion of Proposition \ref{wr}, $\widehat{\otimes}_{\pi,i=1}^n C(K_i)$ is $c_0$-saturated.

\end{proof}

\end{document}